\documentclass[11pt]{article}
\usepackage{amssymb, amsthm, amsmath, amscd}
\newtheorem{thm}{Theorem}[section]
\newtheorem{lem}[thm]{Lemma}
\newtheorem{prop}[thm]{Proposition}
\newtheorem{cor}[thm]{Corollary}
\newtheorem{NN}[thm]{}
\theoremstyle{definition}\newtheorem{df}[thm]{Definition}
\theoremstyle{definition}
\theoremstyle{definition}\newtheorem{exm}[thm]{Example}

\renewcommand{\phi}{\varphi}

\newcommand{\N}{\mathbb{N}}
\newcommand{\Z}{\mathbb{Z}}
\newcommand{\Q}{\mathbb{Q}}
\newcommand{\R}{\mathbb{R}}

\newcommand{\T}{\mathbb{T}}

\newcommand{\Aff}{\operatorname{Aff}}

\newcommand{\hm}{homomorphism}
\newcommand{\dt}{\delta}
\newcommand{\ep}{\epsilon}

\newcommand{\andeqn}{\,\,\,{\rm and}\,\,\,}
\newcommand{\rforal}{\,\,\,{\rm for\,\,\,all}\,\,\,}
\newcommand{\CA}{$C^*$-algebra}
\newcommand{\SCA}{$C^*$-subalgebra}

\newcommand{\af}{{\alpha}}
\newcommand{\bt}{{\beta}}

\newcommand{\beq}{\begin{eqnarray}}
\newcommand{\eneq}{\end{eqnarray}}
\newcommand{\tforal}{\,\,\,\text{for\,\,\,all}\,\,\,}
\newcommand{\tand}{\,\,\,\text{and}\,\,\,}

\newcommand{\Om}{\Omega}

\title{Minimal dynamical systems on
connected odd dimensional spaces}
\author{Huaxin Lin
 }
\date{}
\begin{document}

\maketitle

\begin{abstract}
Let $\bt: S^{2n+1}\to S^{2n+1}$ be a minimal homeomorphism ($n\ge 1$). We show that
the crossed product $C(S^{2n+1})\rtimes_\bt \Z$ has rational tracial rank at most one.
More generally, let $\Om$ be a connected compact metric space with finite covering dimension and
with $H^1(\Om, \Z)=\{0\}.$ Suppose that $K_i(C(\Om))=\Z\oplus G_i$ for some finite abelian group $G_i,$ $i=0,1.$ 
Let $\bt: \Om\to\Om$ be a minimal homeomorphism. We also show that
$A=C(\Om)\rtimes_\bt\Z$ has rational tracial rank at most one and is
classifiable.
In particular, this applies to the minimal dynamical systems on
odd dimensional real projective spaces.
This was done by studying the minimal homeomorphisms on $X\times \Om,$ where
$X$ is the Cantor set.
\end{abstract}

\section{Introduction}
Let $\Om$ be a compact metric space and let $\af: \Om\to \Om$ be a minimal homeomorphism. 
We study the  resulting crossed product \CA\, $C(\Om)\rtimes_\af\Z.$  There are
interesting interplays between minimal dynamical systems and the study of \CA s.
A classical result of Giordano, Putnam and Skau \cite{GPS} showed
that two Cantor minimal systems are strong orbit equivalent if and only if the associated crossed
product \CA s are isomorphic. The \CA\, theoretic aspect of their result is indebted to
the fact that the crossed product \CA s are unital simple A$\T$-algebras of real rank zero and
belong to the class of classifiable \CA s (see Definition \ref{DfA}  below).  Therefore these \CA s  are classified up to isomorphisms by their Elliott invariant,
namely, in this case, by their ordered $K$-theory. In turn, the Cantor minimal systems are classified up to
strong orbit equivalence by their induced $K$-theory.   \CA s of the form $C(\Om)\rtimes_\af\Z$ are always
simple  when $\af$ are minimal. These \CA s  provide a rich source of unital separable amenable simple \CA s which satisfy
the so-called Universal Coefficient Theorem. On the other hand, the rapidly developing Elliott  program, otherwise known as the program of classification of amenable
\CA s by $K$-theoretical invariants, provides a possible way to characterize minimal dynamical
systems by their $K$-theoretical invariants.  It is therefore important to know
when $C(\Om)\rtimes_\af\Z$ belongs to the classifiable class of amenable simple \CA s (in the sense of the Elliott program).  Elliott and Evans (\cite{EE}) showed that all irrational rotation algebras,
which are crossed product \CA s from  minimal dynamical systems on the circle, are classifiable. In fact, they are unital $A\T$-algebras of real rank zero.  Let $A=C(\Om)\rtimes_\af\Z.$ In \cite{LP}, it is shown that, if
$\Om$ has finite covering dimension and
$\rho_A(K_0(A)),$ the tracial image of $K_0(A),$ is dense in $\Aff(T(A)),$ the space of
real continuous affine functions on the tracial state space, then $A$ has tracial rank zero
(the converse also holds). Consequently  $A$ is classifiable.  In the case that $\Om$ is connected and $(\Om, \af)$ is unique
ergodic, this result states that if $(\Om, \af)$ has an irrational rotation number, then
$A$ is classifiable. In particular, this recovers the previously mentioned  case of  irrational
rotations on the circle.  With recent  developments in
the Elliott program (\cite{LN}, \cite{W1}  and \cite{Lnappen}), the classifiable
\CA s now  include the class of \CA s
which are rationally of tracial rank at most one.  The result of \cite{LP} was  pushed  further by
Toms and Winter to 
great generality: if projections in $A$ separate the tracial states,
then $A$ has rational tracial rank zero. Moreover these crossed products are also classifiable
by the Elliott invariant (see \cite{TW}).

On the other hand, during these developments, minimal dynamical systems on $X\times \T,$
where $X$ is the Cantor set and $\T$ is the circle, have been studied (see %\cite{LM1},
\cite{LM2}
and \cite{LM3}).  More general cases were also studied in \cite{Su}.   In both \cite{LP} and
\cite{TW}, the Putnam algebra $A_x$ was used as the main bridge. One may view $A_x$ as a
large \SCA\, of $A=C(\Om)\rtimes_\af\Z$ by ``taking away" one point. In \cite{LM2} and \cite{LM3},
a smaller \SCA\, is used. That \SCA\, may be viewed as a  \SCA\, by ``taking  away"
a circle (an idea of Hiroki Matui).   This method seemed to be too special to be useful in more general cases. However, it has been recently adopted by K. Strung  \cite{KS} to obtain very interesting result about crossed products
of  certain minimal systems on the odd spheres  $S^{2n+1}. $  Strung showed that, by studying
the minimal systems of product type $X\times S^{2n+1},$ one can provide examples
of non-unique ergodic minimal dynamical systems on the odd spheres whose crossed product \CA s are classifiable.
Let $A$ be the crossed product obtained from the minimal system on the odd sphere.
It should be noted $A$ may not have rational tracial rank zero. In particular,
its projections may not separate the tracial states.   Nevertheless, $A$ has rational tracial rank at most one, i.e., $A\otimes U$ has tracial rank at most one for any UHF-algebra $U$ of infinite type.
Therefore a more general classification result in \cite{Lninv} can be applied.

  In \cite{KS}, the minimal homeomorphisms on the odd spheres are  assumed
  to be limits of periodic homeomorphisms
  constructed by ``fast approximation-conjugation".   In this note we will study the general minimal dynamical systems
 on the odd spheres as well as on the odd dimensional real projective spaces. We show
 that crossed product \CA s from any minimal dynamical systems on odd spheres or odd dimensional
 real projective spaces have rational tracial rank at most one and are classifiable. It should be mentioned that there are no minimal homeomorphisms on even spheres or on the even dimensional real projective spaces because of the existence of fixed points.  We actually prove much more general results (see \ref{MT2} and 
 \ref{MT3} below).
 %and \ref{MT4}).
The methods we used here to study the minimal dynamical systems on the product spaces of the form $X\times \Om,$ where $X$ is the Cantor set and $\Om$ is a connected space are those
developed in \cite{LM2} and \cite{LM3}.  We also use a recent uniqueness theorem (see \ref{TUni} below)
from \cite{LnHom}.  We continue to use the argument of Strung as well as an embedding
result of Winter  (\cite{W2}). The classification result in \cite{Lninv} is also applied.

This note is organized as follows. Section 2 serves as preliminaries of this note. In Section 3,
we study the general minimal dynamical systems on the product spaces $X\times \Om,$
where $X$ is the Cantor set and $\Om$ is a connected compact space.  Examples of minimal dynamical systems studied in Section 3 are presented in Section 4. Applications are presented
in Section 5.

{\bf Acknowledgements}  This work was done during the author's stay at the Research Center
for Operator Algebras at East China Normal University.

\section{Preliminaries}

\begin{df}\label{Dat}
{\rm Let $A$ be a unital \CA. Denote by $U(A)$ the unitary group of $A$ and
$U_0(A)$ the connected component of $U(A)$ containing the identity.
Denote by $CU(A)$ the closure of the commutator subgroup of $U_0(A).$
Denote by $T(A)$ the tracial state space of $A.$
We also use $T(A)$ for traces of the form
$\tau\otimes {\rm Tr}$ on $M_n(A)$ for all integer $n,$ where ${\rm Tr}$ is the standard (un-normalized) trace on $M_n.$
Denote by $\rho_A: K_0(A)\to \Aff(T(A))$ defined by
$\rho_A([p])=\tau(p)$ for all projections in $M_n(A),$ $n=1,2,....$
}
\end{df}

\begin{df}\label{Ddet}
{\rm
Let $A=M_n$ and $a\in A.$ We use ${\rm det}(a)$ for the usual determinant.
If $A=C(X)\otimes M_n,$ we will often identify $A$ with $C(X, M_n).$
If $f\in A,$ we use ${\rm det}(f)$ for the function ${\rm det}(f)(x)$ in $C(X).$

Let $A$ be a unital \CA\, with $T(A)\not=\emptyset$ and let $u\in U_0(A).$
Let $\{u(t): t\in [0,1]\}\subset U_0(A)$ be a piecewise smooth continuous path
with $u(0)=u$ and $u(1)=1_A.$
Define
\beq\label{Ddet-1}
{\tilde \Delta}_\tau(u(t))={1\over{2\pi i}}\int_0^1 \tau({du(t)\over{dt}}u(t)^*)dt\rforal \tau\in T(A).
\eneq
If $u\in CU(A),$ then ${\tilde \Delta}_\tau(u(t))\in \overline{\rho_A(K_0(A))}.$
This de La Harp-Skandalis determinant (which is independent of the choice of the path)
  gives
a  \hm
\beq\label{Ddet-2}
\Delta: U_0(A)/CU(A)\to \Aff(T(A))/{\overline{\rho_A(K_0(A))}}
\eneq
(see \cite{KT}).

If $u\in U(A),$ we will use $\overline{u}$ for its image in $U(A)/CU(A).$
}
\end{df}

\begin{df}\label{Dphb}
{\rm
Let $\Om$ be a compact metric space. Denote by ${\rm Homeo}(\Om)$ the set of all homeomorphisms on $\Om$ equipped with the topology of pointwise convergence.
Let $\bt\in {\rm Homeo}(\Om).$ Denote by
${\tilde \bt}: C(\Om)\to C(\Om)$ the automorphism defined by
${\tilde \bt}(f)=f\circ\bt^{-1}$ for all $f\in C(\Om).$
If $F\subset \Om,$ denote by $\chi_F$ the characteristic function of $F.$
When $F$ is a clopen set, $\chi_F\in C(\Om).$

}
\end{df}

\begin{lem}{\rm (Lemma 2.1 of \cite{LM2})}\label{LXom}
Let $X$ be the Cantor set and $\Om$ be a connected compact metric space.
Let $\bt\in {\rm Homeo}(X\times \Om).$
Then there is $\gamma\in {\rm Homeo}(X)$ and a continuous map
$\phi: X\to {\rm Homeo}(\Om)$ such that $\bt(x, \xi)=(\gamma(x), \phi_x(\xi))$ for all $(x, \xi)\in X\times \Om.$
\end{lem}

\begin{proof}
Let $p_X: X\times \Om\to X$ and $p_\Om: X\times \Om\to \Om$ be
projection maps such that $p_X(x, \xi)=x$ and $p_\Om(z, \xi)=\xi$
for all $(x, \xi)\in X\times \Om.$
Fix $(x, \xi)\in X\times \Om.$ Then $\{x\}\times \Om$ is the connected component
of $X\times \Om$ containing $(x, \xi).$ The homeomorphism $\bt$ maps
it into a connected component containing $\bt(x, \xi)=(x_1, \xi_1),$
$x_1=p_X(\bt(x,\xi))$ and $\xi_1=p_\Om(\bt(x, \xi)).$
Therefore as just mentioned, the component is  $\{x_1\}\times \Om.$
Define $\gamma(x)=p_X(\bt(x, \xi))$ for $x\in X.$  Since $p_X(\bt(x, \xi))=
p_X(\bt(x, \xi')),$ it is a well-defined map. Since $\bt$ is a homeomorphism,
it is also ready to see that $\gamma\in {\rm Homeo}(X).$
Fix $x\in X,$ then map $\phi_x(\xi)=p_\Om(\bt(x, \xi))$ is a homeomorphism
from $\Om$ onto $\Om.$ It is easy to check that $\phi: X\to {\rm Homeo}(\Om)$
is continuous. 
\end{proof}

\begin{NN}\label{DKnotation}
{\rm
Let $X$ be the Cantor set and $\Om$ be a connected compact metric space.
Then
$$
K_0(C(X\times\Om))=C(X, K_0(C(\Om)))\andeqn K_1(X\times \Om)=C(X, K_1(C(\Om))),
$$
where the group $K_i(C(\Om))$ is viewed as a discrete space, $i=0,1.$
Moreover
$$
K_0(C(X\times \Om))_+=\{f\in C(X, K_0(C(\Om))): f(x)\in K_0(C(\Om))_+\}=C(X,K_0(C(\Om)))_+,
$$
Let $(X, \af)$ be a Cantor minimal system.
% and let $\phi: X\to Homeo(\Om)$ be a continuous map.
%We make assumption that
%$[\phi_x]=[{\rm id}_{C(\Om)}$ for each $x\in X.$
Denote
\beq\label{DKn-0}
K^0(X,\af, K_0(C(\Om)))=C(X,K_0(C(\Om)))/\{f-f\circ\af^{-1}: f\in C(X, K_0(C(\Om)))\}
\eneq
equipped with the positive cone
\beq\label{DKn-0+}
K^0(X, \af, K_0(C(\Om)))_+=\{[f]: f\in C(X,K_0(C(\Om)))_+\}.
\eneq
Denote
\beq\label{DKn-1}
{\rm ker}({\rm id}_X-\af^{-1})_{*i}&=&\{f\in C(X, K_i(C(\Om))): f-f\circ \af^{-1}=0\}.
\eneq
Denote
\beq\label{DKn-0}
K^i(X,\af\times \phi, K_0(C(\Om)))=C(X,K_i(C(\Om)))/\{f-f\circ (\af\times \phi)^{-1}: f\in C(X, K_i(\Om))\}
\eneq
with $K^0(X,\af\times \phi, K_0(C(\Om)))$ equipped with the positive cone
\beq\label{DKn-0+}
K^0(X, \af\times \phi, K_0(C(\Om)))_+=\{[f]: f\in C(X,K_0(C(\Om)))_+\}.
\eneq
For $i=0,1,$ denote
\beq\label{DKn-1}
{\rm ker}({\rm id}_{X\times \Om}-(\af\times \phi)^{-1})_{*i}&=&
\{f\in C(X, K_i(C(\Om))): f-f\circ (\af\times \phi)^{-1}=0\}.
\eneq

}
\end{NN}

\begin{prop}\label{AKT}
Let $(X, \af)$ be a Cantor minimal system, $\Om$ be a connected compact metric space and let $\phi: X\to {\rm Homeo}(\Om).$
%such that $(\phi_x)_{*0}={\rm id}_{K_0(C(\Om))}.$
Let $A=C(X\times \Om)_{\af\times \phi}\Z.$ Then there are short exact sequences
\beq\label{AKT-1}
&&0\to K^0(X,\af\times \phi, K_0(\Om))\to K_0(A)\to {\rm ker} ({\rm id}_{X\times\Om}-(\af\times\phi)^{-1}_{*1})\to 0\tand\\\label{AKT-2}
&&0\to K^1(X, \af\times \phi, K_1(\Om))\to K_1(A)\to  {\rm ker}({\rm id}_{X\times \Om}-(\af\times \phi)^{-1}_{*0})\to 0.
\eneq
Moreover,
\beq\label{AKT-3}
\rho_A(K^0(X,(\af\times \phi)^{-1}, K_0(\Om)))=\rho_A(K^0(X, \af,\Z)).
\eneq
Furthermore, if $(\phi_x)_{*i}={\rm id}_{K_i(C(\Om))}$ for all $x\in X,$ $i=0,1,$ then
\beq\label{AKT-3+}
&&0\to K^0(X,\af\times {\rm id}, K_0(\Om))\to  K_0(A)\to K_1(C(\Om))\to 0\tand\\\label{AKT-3++}
&&0\to K^1(X, \af\times {\rm id}, K_1(\Om))\to K_1(A)\to  K_0(C(\Om))\to 0.
\eneq
\end{prop}

\begin{proof}
It is clear that (\ref{AKT-1}) and (\ref{AKT-2}) follow from the Pimsner-Voiculescu six-term exact sequence directly.
To see the second statement, we note that
$\Om$ is connected and all tracial states on $K_0(C(\Om))$ agree with the rank.
%By the assumption that $(\phi_x)_{*0}={\rm id}_{K_0(C(\Om))}$ and $\af$ is minimal,
%$${\rm ker} (\id_{K_0(C(X, K_0(C(\Om)))}-(\af\times\phi^{-1})_{*0})$$
 %is the subgroup of constant maps in $C(X, K_0(C(\Om)))$ which is isomorphic
%to $K_0(C(\Om)).$ Thus, again, by Pimsner-Voiculescu six-term exact sequence, (\ref{AKT-2}) follows.
To see the last statement,
we note that
$$
K_i(C(X\times \Om))=C(X, K_i(C(\Om))),\,\,\,i=0,1.
$$
In the case $(\phi_x)_{*i}={\rm id}_{K_i(C(\Om))},$ $i=0,1,$ (only) constant elements are invariant under $(\af\times \phi)^{-1}_{*i}.$ It follows that
$$
{\rm ker} ({\rm id}_{X\times\Om}-(\af\times \phi)^{-1}_{*i})=K_{i+1}(C(\Om)),\,\,\,i=0,1.
$$

\end{proof}

\begin{lem}\label{Ltrace}
Let $\Om$ be a compact metric space with $U(C(\Om))=U_0(C(\Om))$ and $\bt\in {\rm Homeo}(\Om).$
Then
\beq\label{Ltrace-1}
\rho_B(K_0(B))=\rho_A(\imath_{*0}(K_0(C(\Om)))),
\eneq
where $B=C(\Om)\rtimes_\bt \Z$ and $\imath: C(\Om)\to B$ is the natural
embedding.
Consequently,
in \ref{AKT}, if $H^1(\Om, \Z)=\{0\},$ then
$$
\rho_A(K_0(A))=\rho_A(K^0(X, \af, \Z)).
$$
\end{lem}

\begin{proof}
The first part follows from Chapter VI of \cite{Exel1} (see also 10.10.5 of \cite{Bl}). For the second part, we note that $X$ has zero dimension
so $H^1(X\times \Om,\Z)=\{0\}.$ Therefore, by (\ref{AKT-3}),
\beq\label{Ltrace-2}
\rho_A(K_0(A))=\rho_A(K^0(X,(\af\times \phi)^{-1}, K_0(C(\Om))))=\rho_A(K^0(X, \af, \Z)).
\eneq
\end{proof}

\begin{df}\label{DfA}
{\rm
Denote by ${\cal A}$ the class of unital ${\cal Z}$-stable separable simple amenable \CA s which satisfy the
Universal Coefficient Theorem and which have  rational tracial rank at most one, i.e.,
 $A\otimes U$ has tracial rank at most one, where $U$ is any infinite dimensional
UHF-algebra (see \cite{Lninv}). This class of \CA s
is classifiable in the sense that, if $A, \, B\in {\cal A},$ then $A\cong B$ if and only if
they have the isomorphic Elliott invariant (\cite{Lninv}). This class contains all unital simple AH-algebras
with no dimension growth  as well as the Jiang-Su algebra.  A description of the range
of the invariant is presented in \cite{LN2}.

}
\end{df}

\section{The \SCA\, $A_x$}

\begin{df}\label{DAx}
{\rm
Let $(X,\af)$ be a Cantor minimal system and $\Om$ be a connected finite dimensional
compact metric space.
Fix $x\in X.$ Denote by $A_x$ the \SCA\, of $A=C(X\times \Omega)\times_{\af\times \phi} \Z$ generated by $C(X\times \Om)$ and $uC_0(X\setminus\{x\}\times \Om),$ where
$u$ is the unitary in $A$ which implements the action $\af\times \phi.$
}
\end{df}

\begin{thm}\label{TAx}
The \CA\, $A_x$ is isomorphic to a unital simple AH-algebra with slow dimension growth.
%Moreover, $T(A_x)=T(A).$
\end{thm}

\begin{proof}
Part of the proof is known. As in the proof of the part (5) of Proposition 3.3 of \cite{LM2}, using
groupoid \CA s, $A_x$ is simple,  since we assume that $\af\times \phi$
is  minimal.

We now assume that the dimension of $\Om$ is $d.$
To show $A_x$ is locally AH, we use an argument of Ian Putnam (3.1 of \cite{Pu1}) and proceed as in the proof of Proposition 3.3 of \cite{LM2}.
Let
$$
{\cal P}_n=\{X(n,v, k): v\in V_n, k=1,2,...,k(v)\}
$$
be a sequence of Kakutani-Rohlin partitions which gives a Bratteli-Vershik model
for $(X, \sigma)$ (see Theorem 4.2 of \cite{HPS} or Section 2 of \cite{M3}). We also assume the roof sets
$$
R({\cal P}_n)=\bigcup_{v\in V} X(n,v,h_n(v))
$$
shrink to a single point $x.$ Let $A_n$ be the \CA\, generated by
$C(X\times\Om)$ and $uC({\cal R}({\cal P}_n)^c\times \Om).$
Since $R({\cal P}_{n+1})\subset R({\cal P}_n),$ $A_n\subset A_{n+1},$
$n=1,2,....$
It is easy see that $A_x$ is the norm closure of the union of all $A_n$' s.
By using a similar argument to Lemma 3.1 of \cite{Pu1}, it can be shown that
$A_n$ is isomorphic to
\beq\label{TAx-2}
\bigoplus_{v\in {\cal V}_n} M_{h_n(v)}\otimes C(X(n,v,h_n(v)))\otimes C(\Om)\cong
\bigoplus_{v\in {\cal V}_n}M_{h_n(v)}(C(Y_{n,v})),
\eneq
where $Y_{n,v}$ is a compact metric space of covering dimension $d.$
In fact, let
$$
p_v=\sum_{k=1}^{h_n(v)}\chi_{X(n,v,k)}
$$
for each $v\in V.$  It is easy to check that $p_vu(1-\chi_{R({\cal P}_n)})=u(1-\chi_{R({\cal P}_n)})p_v.$ It follows that $p_v$ is central in $A_n.$  Put $e_{i,j}(n,v)=u^{i-j}\chi_{X(n,v,j)\times \Om}.$ One  
that, for each $n$ and $v,$  $\{e_{i,j}(n,v)\}_{i,j}$ ($1\le i,\,j\le h_n(v)$)
form matrix units in $A_n$ with
$$
\sum_{i=1}^{h_n(v)}  e_{i,i}=\sum_{i=1}^{h(v)}\chi_{X(n,v,i)}=p_{n,v}.
$$
Note that
\beq\label{TAx-4}
\chi_{X(n,v,h_n(v))}A_n\chi_{X(n,v,h_n(v))}=C(X(n,v,h_n(v)))\times \Om).
\eneq
One checks that the \SCA\, generated by $\{e_{i,j}(n,v)\}$ and $\chi_{X(n,v,h_n(v))}A_n\chi_{X(n,v,h_n(v))}$ is isomorphic to $M_{h_n(v)}(C(X(n,v,h_n(v))\times \Om)).$
Therefore (\ref{TAx-2}) holds.
It follows from Theorem 1.1 of \cite{Lnloc} that  $A_x$ is  a unital simple AH-algebra with no dimension growth.

%To show $T(A_x)=T(A),$ it suffices to show that every tracial state
%$\tau\in T(A_x)$ can be extended to tracial state of $A.$
%The proof can be  proceed exactly the same way as that   of part (4) of Proposition 3.3 of \cite{LM2}. We omit the details.
\end{proof}

\begin{prop}\label{AxKT}
Let $(X, \af)$ be a Cantor minimal system, let $x\in X,$ let $\Om$ be a connected compact metric space and $\phi: X\to {\rm Homeo}(\Om)$ be a continuous map.
%Suppose that $(\phi_x)_{*0}={\rm id}_{K_0(\Om)},$ $i=0,1$ for each $x\in X.$
Then, for $i=0,1,$ 
\beq\label{AxKT-1}
%&&\hspace{-0.4in}K_0(A_x)\cong K^0(X, \af\times \phi, K_0(C(\Om)))\andeqn\\\label{AxKT-1+}
K_i(A_x)\cong K_i(C(X\times \Om))/\{f-f\circ(\af\times \phi)_{*i}^{-1}: f(x)=0, f\in C(X,K_i(C(\Om)))\}.
\eneq
Moreover, the imbedding $\imath: A_x\to A$
gives
an affine homeomorphism $\imath_{\sharp}: T(A)\to T(A_x)$ and
gives an order isomorphism $\rho_{A_x}(K_0(A_x))=\rho_A(K^0(X,\af,\Z)).$
Moreover, if $q\in C(X)$ is a projection, then
$uqu^*$ and $q$ are equivalent in $A_x.$
\end{prop}

\begin{proof}
Let $A_n$ be in the proof of \ref{TAx}, i.e., $A_n$ is generated by
$C(X\times \Om)$ and $uC(R({\cal P})^c\times \Om).$ There is a natural
\hm\, from $K_i(C(X\times \Om))$ to $K_i(A_n)$ (by the embedding
of $C(X\times \Om)$). By (\ref{TAx-2}), this \hm\, is surjective.
 For $i=0,1,$ the kernel is
$$
\{f-f(\af\times \phi)_{*1}^{-1}: f(y)=0 \tforal y\in R({\cal P}_n),\, f\in K_i(C(X\times \Om))\}.
$$
Thus (\ref{AxKT-1}) holds.

We now prove that, for any projection $q\in C(X),$ $uqu^*$ and $q$ are equivalent
in $A_x.$  If $q(x)=0,$ then $uq\in A_x.$ It follows that $qu^*\in A_x.$
Therefore $uqu^*$ and $q$ are equivalent in $A_x.$
Suppose that $q(x)\not=0.$
It suffices to show that
\beq\label{AxKT-3}
[q]-[q\circ (\af\times \phi)^{-1}]=0\,\,\,{\rm in}\,\,\, K_0(A_x).
\eneq
Let
$f(y)=[q(y)]-[1_{C(X\times \Om)}(y)]$ for all $y\in X.$ Note that
$f\in C(X, K_0(C(\Om)))$ with $f(x)=0.$
Note that $ 1_{C(X\times \Om)}-1_{C(X\times \Om)}\circ (\af\times\phi)^{-1}=0.$
It follows
that
\beq\label{AxKT-4}
[q]-[q\circ (\af\times \phi)^{-1}]=
f-f\circ (\af\times \phi)^{-1}
+[1_{C(X\times \Om)}]-[1_{C(X\times \Om)}\circ (\af\times\phi)^{-1}].\\
\eneq
From (\ref{AxKT-1}), this implies that $[q]-[q\circ (\af\times \phi)^{-1}]=0$
in $K_0(A_x).$ This proves that $q$ and $uqu^*$ are equivalent in $A_x.$

To show $T(A_x)=T(A),$ it suffices to show that every tracial state
$\tau\in T(A_x)$ can be extended to tracial state of $A.$
Let $U$ be a clopen neighborhood of $x$ such that
$U, \af^{-1}(U), \cdots, \af^{-n}(U)$ are mutually disjoint. Let $p=\chi_{U\times \Om}.$ We have shown $p,$ $upu^*,$ $u^npu^{*n}$ are mutually equivalent in $A_x$ and mutually orthogonal.
The proof that $T(A_x)=T(A)$ can then  proceed exactly the same way as that   of part (4) of Proposition 3.3 of \cite{LM2}.

To show
that $\rho_{A_x}(K_0(A_x))=\rho_A(K^0(X, \af, \Z)),$ we first note that we have just proved that
the map sending $[\chi_O]$ to $[\chi_O]$ (for clopen sets $O\subset X$)  is an embedding from
$K^0(X, \af, \Z)$ into $K_0(A_x).$
  Since $\Om$ is connected, the subgroup $K^0(X, \af, \Z)$
injectively maps into $K^0(X, \af\times \phi, K_0(C(\Om)))\subset K_0(A).$

\end{proof}

\section{Tracial rank}

\begin{df}\label{DLmeas}
{\rm
%Let $A$ be a unital simple \CA\, and let $T(A)$ be the tracial state space.
Let $T$ be a compact Choquet simplex. Suppose that $Y$ is a compact metric space and $L: C(Y)_{s.a.}\to \Aff(T)$ is an affine \hm. We say $L$ is unital and strictly positive if $L(1_{C(Y)})(\tau)=1$ for all $\tau\in T,$ and
$L(f)(\tau)>0$ for all $\tau\in T$ if $f\not=0$ and $f\ge 0.$

Suppose that $L: C(Y)_{s.a.}\to \Aff(T)$ is a strictly positive affine
\hm. Let $f\not=0$ and $f\ge 0.$
Then, since $T(A)$ is compact,
\beq\label{DLm-1}
\inf\{L(f)(\tau): \tau\in T(A)\}>0.
\eneq
For each open subset $O\subset Y,$ let
\beq\label{DLm-2}
d(O)=\inf_{\tau\in T}\{\sup\{ L(f)(\tau): 0\le f\le 1,\, {\rm supp}(f)\subset O\}\}.
\eneq
Then, for any non-empty open subset $O\subset Y,$
$d(O)>0.$
For each $a\in (0,1),$ let $\{x_1, x_2,..., x_m\}\subset Y$ be an $a/4$-dense subset.
Define
\beq\label{DLm-3}
D(a,i)=d(B(x_i,a/4)),\,\,\,i=1,2,...,m.
\eneq
Put
\beq\label{DLm-4}
\bigtriangledown_0(a)=\min\{D(a,i): i=1,2,...,m\}.
\eneq
For any $x\in Y,$ there exists $i$ such that $B(x,a)\supset B(x_i, a/4).$
Thus
\beq\label{DLm-5}
d(B(x,a))\ge (3/4)\bigtriangledown_0(a).
\eneq
Put $\bigtriangledown(a)=(3/4)\bigtriangledown_0(a)$ for all $a\in (0,1).$
Now let $A$ be a unital separable simple \CA\, with $T(A)=T$ and
let $\phi: C(Y)\to A$ be a unital monomorphism. Then
$\phi_\sharp: C(Y)_{s.a.}\to \Aff(T(A))$ defined by
$$
\phi_{\sharp}(f)(\tau)=\tau\circ \phi(f)\rforal f\in C(Y)_{s.a.}
$$
is a unital strictly positive affine \hm.
It is easy to check that, for any $1>\sigma>0,$
there is a finite subset ${\cal H}\subset C(Y)_{s.a.}$
and $\eta>0$ such that
\beq\label{DLm-6}
\mu_{\tau\circ \phi}(B(x,r))\ge \bigtriangledown(r)
\eneq
for all open ball with radius $r\ge \sigma,$ provided
that
\beq\label{DLm-7}
|\tau\circ \phi(g)-L(g)(\tau)|<\eta\rforal g\in {\cal H},
\eneq
where $\mu_{\tau\circ \phi}$ is the Borel probability measure induced by
the state $\tau\circ \phi.$
}
\end{df}

Let $B$ be a unital \CA\, with $T(B)\not=\emptyset.$ Note that, by \cite{KT}, there is a splitting short exact sequence:
$$
0\to \Aff(T(B))\to \cup_{n=1}^{\infty}U(M_n(B))/CU(M_n(B))\to K_1(B)\to 0.
$$
Fix a splitting map, we denote by $U_c(K_1(B))$ the image of $K_1(B)$ under the splitting map.  

\vspace{0.2in}

We will use the following uniqueness theorem:

\begin{thm}\label{TUni} {\rm (Theorem 5.9 of \cite{LnHom})}
Let $Y$ be a compact metric space and let
%$A$ be a unital simple
%\CA\, with $TR(A)\le 1.$
$T$ be a compact Choquet simplex.
Suppose that $L: C(Y)_{s.a.}\to \Aff(T(A))$ is a unital strictly positive affine \hm.
Let $\ep>0$ and ${\cal F}\subset C(Y)$ be a finite subset.
There exists  a finite subset ${\cal H}\subset C(Y)_{s.a.},$ a finite subset ${\cal P}\subset \underline{K}(C(Y)),$ a finite subset
${\cal U}\subset U_c(K_1(C(Y))),$ $\dt>0$ and $\eta>0$ satisfying the following:
Suppose that $\phi_1, \phi_2: C(Y)\to A$ are two unital monomorphisms
for some unital simple \CA\, $A$ of tracial rank at most one  with $T(A)=T$ such that
\beq\label{Tuni-1}
[\phi_1]|_{\cal P}&=&[\phi_2]|_{\cal P},\\
|\tau\circ \phi_i(g)-L(g)(\tau)|&<&\dt\tforal g\in {\cal H},\,\,i=1,2,\tand\\
{\rm dist}({\overline{\phi_1(v)\phi_2(v^*)}}, \overline{1_A})&<&\eta\tforal v\in {\cal U}.
\eneq
Then there is a unitary $w\in A$ such that
\beq\label{Tuni-2}
\|w^*\phi_1(f)w-\phi_2(f)\|<\ep\tforal f\in {\cal F}.
\eneq
\end{thm}

\begin{proof}
This follows immediately from Theorem 5.9 of \cite{LnHom} and the discussion above in \ref{DLmeas}.

\end{proof}

The following is  well-known.

\begin{lem}\label{Ldet}
Let $Y$ be a compact metric space such that $U(C(Y))=U_0(C(Y)).$
Then, for any $z\in K_1(C(Y)),$ there is  an integer $m\ge 1$ and a unitary
$v\in M_m(C(Y))$ such that
$$
[v]=z\andeqn {\rm det}(v)(y)=1\tforal y\in Y.
$$
\end{lem}

\begin{proof}
There exists an integer $m\ge 1$ and a unitary $w\in M_m(C(Y))$ such that
$[w]=z.$
Put
\beq\label{Ldet-1}
w_{00}(y)={\rm det}(w)(y)\tforal y\in Y.
\eneq
Then $w_{00}\in U(C(Y))=U_0(C(Y)).$
It follows that
\beq\label{Ldet-2}
w_0(y)=\begin{pmatrix} w_{00}^*(y) &0&\cdots & 0\\
                                       0 & 1 &\cdots &0\\
                                       && \ddots  & &\\
                                         &&&1\\
                                       \end{pmatrix}\in U_0(M_m(C(Y))).
                                      \eneq
Define $v=w_0w.$ Then $[v]=[w]=z$ and
\beq\label{Ldet-3}
{\rm det}(v)(y)={\rm det}(w_0)(y){\rm det}(w)(y)=1\rforal y\in Y.
\eneq
\end{proof}

\begin{lem}\label{MtechL}
Let $(X, \af)$ be a Cantor minimal system, let $\Omega$ be a  connected finite dimensional compact metric space with $U(C(\Om))=U_0(C(\Om))$
and let $\phi: X\to {\rm Homeo}(\Omega)$ be a continuous map. Put $A=C(X\times \Om)\rtimes_{\af\times\phi}\Z.$
Suppose that
there is $x\in X$ and an integer $k\ge 1$ such that
\beq\label{MtechL-1}
[\Phi_y]=[{\rm id}_{C(\Om)}]\,\,\,{\rm in}\,\,\, KL(C(\Om),C(\Om)),
\eneq
where $\Phi_y: C(\Om)\to C(\Om)$ is defined by
\beq\label{MtechL-1+}
\Phi_y(f)=f\circ \phi_{\af^{-k+1}(y)}^{-1}\circ \phi_{\af^{-k+2}(y)}^{-1}\circ\cdots \circ \phi_y^{-1} \tforal f\in C(\Om)
\eneq
 and for all
$y\in \{\af^{j}(x): j\in \Z\}$ and $\af^k$ is minimal.
%  and
%\beq\label{MtechL-2}
%U(C(\Om))=U_0(C(\Om)).
%\rho_A(K_0(A))=\rho_A(K^0(X, \af, K_0(C(\Om)))).
%\eneq
Let $x\in X.$  Then, for any $N\in \N,$ $\ep>0,$ and any finite subset ${\cal F}\subset C(X\times \Om),$ there is an integer $M>N,$  a clopen neighborhood $O$ of $x$ and partial isometry $w\in A_x$ which satisfy the following:

{\rm (1)} $\af^{-N}(O), \af^{-N+1}(O),..., O, \af(O),...,\af^M(O)$ are mutually disjoint and $\mu(O)<\ep/M$ for every $\af$-invariant probability measure $\mu;$

{\rm (2)} $w^*w=\chi_O$ and $ww^*=\chi_{\af^M(O)};$

{\rm (3)} $u^{*i}wu^i\in A_x$ for $i=0,1,...,N-1;$

{\rm (4)} $\|wf-fw\|<\ep$ for all $f\in {\cal F}.$

\end{lem}

\begin{proof}
Since $U(C(\Om))$ is connected and $X$ has zero-dimensional,
one has
\beq\label{MtechLn-1}
U(C(X\times \Om))=U_0(C(X\times \Om)).
\eneq
It follows from \ref{Ltrace} that
\beq\label{MtechLn-2}
\rho_A(K_0(A))=\rho_A(K^0(X, \af, \Z)).
\eneq
Therefore, by \ref{AxKT},
% and \ref{TAx},
the embedding $\imath: A_x\to A$ gives
\beq\label{MtechLn-3}
\rho_{A_x}(K_0(A_x))=\rho_A(K_0(A)).
\eneq
Note $A_x$ is a unital simple AH-algebra with no dimension growth, by \ref{TAx}. So $TR(A_x)\le 1.$ It is generated by
$C(X\times \Om)$ and $uC((X\setminus \{x\})\times \Om).$
The $g\to 1_{X}\otimes g$ gives a unital embedding from
$C(\Om)$ into $C(X\times \Om).$ Therefore there is
a unital embedding $\imath: C(\Om)\to A_x.$
Let $L: C(\Om)\to \Aff(T(A_x))$ be the unital strictly positive
affine \hm\, induced by $\imath.$ Note that $L(g)=\tau(1\otimes g)$ for all $g\in C(\Om)_{s.a.}$ and 
for all $\tau\in T(A_x).$

Without  loss of generality, we may assume that
$$
{\cal F}=\{f\otimes 1_{\Om}, 1_X\otimes g: f\in {\cal F}_0\tand g\in{\cal F}_1\},
$$
where ${\cal F}_0\subset C(X)$ and ${\cal F}_1\subset C(\Om)$ are finite subsets. There exists a clopen neighborhood
$B_x$ of $x$ such that
\beq\label{MtechL-4}
|f(x)-f(y)|<\ep/8\rforal y\in B_x\andeqn \rforal f\in {\cal F}_0.
\eneq
Since $\af^k$ is minimal, we can find $n>N$ such that
\beq\label{MtechL-5}
\af^{kn}(x)\subset B_x.
\eneq
Choose a sufficiently small clopen neighborhood $O_x$ of $x$ such that
(1) holds and
\beq\label{MtechL-5}
\af^{kn}(y)\in B_x\rforal y\in O_x.
\eneq
Moreover, by choosing small $O_x,$ we may also require that
\beq\label{MtechL-6}
O_x\cup \af^{kn}(O_x)\subset B_x.
\eneq
Let $p_1=\chi_{O_x}$ and $q_1=\chi_{\af^{kn}(O_x)}.$ Put $M=kn.$
Then
\beq\label{MtechL-7}
u^{M} gp_1u^{*M}=u^Mgu^{*M}q_1\rforal g\in C(\Om).
\eneq
Define
%$$
%\Psi_x
%(f) =f\circ \phi_{\af^{-kn+1}(x)}^{-1}\circ \phi_{\af^{-kn+2}(x)}^{-1}\circ \cdots \circ \phi_{\af^{-1}(x)}^{-1}\circ \phi_x^{-1}.
%$$
%Then
\beq\label{MtechL-7+}
\Psi_x=\Phi_{\af^{(1-k)(n-1)}(x)}\circ \Phi_{\af^{(2-k)(n-2)}(x)}\circ \cdots
\circ \Phi_x.
\eneq
It follows that
\beq\label{MtechL-7+2}
[\Psi_x]=[{\rm id}_{C(\Om)}]\times [{\rm id}_{C(\Om)}]\times\cdots \times [{\rm id}_{C(\Om)}]=[{\rm id}_{C(\Om)}].
\eneq
Let ${\cal H}\subset C(\Om)_{s.a.}$ be a finite subset, ${\cal P}\subset \underline{K}(C(\Om))$ be a finite subset, ${\cal U}\subset
U_c(K_1(C(\Om)))$ be a finite subset, $\dt>0$ and $\eta>0$ be required by
\ref{TUni} for $\ep/4$ (in place of $\ep$) and ${\cal F}_1$ (in place of ${\cal F}$) associated with $L$ given above.
Let ${\cal V}=\{v_1, v_2,....,v_m\}\subset M_K(C(\Om))$ such that
${\cal U}\subset \{\overline{v_i}: 1\le i\le m\}$ and, by \ref{Ldet},
\beq\label{MtechLn-7}
{\rm det}(v_i)(y)=1\tforal y\in \Om,\,\,\, i=1,2,...,m.
\eneq

There is a finite subset ${\cal G}\subset C(\Om)$ and $\dt_1>0$ satisfying the following:
Suppose that $h_1, h_2: C(\Om)\to B$ are  two unital \hm s (for any
unital \CA\, $B$) such that
\beq\label{MtechLnn-1}
\|h_1(g)-h_2(g)\|<\dt_1\rforal g\in {\cal G}.
\eneq
Then
\beq\label{MtechLnn-2}
[h_1]|_{\cal P}=[h_2]|_{\cal P}.
\eneq
Let ${\cal G}_1={\cal G}\cup {\cal H}$ and let
\beq\label{MtechL-7+3}
U={\rm diag}(\overbrace{u,u,...,u}^K).
\eneq

There is a neighborhood $O$ of $x$ with $O\subset O_x$ such that
\beq\label{MtechL-8}
\|u^Mgpu^{*M}-\Psi_x(g)q\|<\min\{\dt,\dt_1, \ep/8\} \rforal g\in {\cal G}_1\andeqn\\\label{MtechL-8+}
\|U^Mv_iPU^{*M}-(\Psi_x\otimes {\rm id}_{M_K})(v_i)Q\|<\eta,\,\,\,1\le i\le m,
\eneq
where
$$
p=\chi_O,\,\,\,q=\chi_{\af^M(O)}, P=p\otimes {\rm id}_{M_K}\andeqn
Q=q\otimes {\rm id}_{M_K}.
$$
Define $\psi_{1,0}, \psi_{2,0}: C(\Om)\to C(X\times \Om)$ by $\psi_{1,0}(f)=f$ (as constant
along $X$) and
$\psi_{2,0}(f)=\Psi_x(f)$ for all $f\in C(\Om).$
It follows from (\ref{MtechL-7+2}) that
\beq\label{MtechL-9}
[\psi_{2,0}]=[\psi_{1,0}]\,\,\, KL(C(\Om), C(X\times \Om)).
\eneq
Define $\psi_1', \psi_2': C(\Om)\to q(C(X\times \Om))q$ by
$\psi_1'(g)=\psi_{1,0}(g)|_{\af^{M}(O)}$ and $\psi_2'(g)=\psi_1'\circ \psi_{2,0}(g)=\Psi_x(g)\cdot q$ for all $g\in C(\Om).$
It follows from (\ref{MtechL-9}) that
\beq\label{MtechL-10}
[\psi_1']=[\psi_2']\,\,\, {\rm in}\,\,\, KL(C(\Om), qC(X\times \Om)q).
\eneq
Denote by $j$ the embedding $qC(X\times \Om)q\to qA_xq,$
$\psi_i=j\circ \psi_i',$ $i=0,1.$
Then
\beq\label{MtechL-11}
[\psi_1]=[\psi_2]\,\,\, {\rm in}\,\,\, KL(C(\Om), qA_xq).
\eneq
It follows from (\ref{MtechL-8}) and \ref{TAx} that
\beq\label{MtechL-12}
|t\circ \psi_1(g)-t\circ\psi_2(g)|<\dt\rforal g\in {\cal H} \,\,\,{\rm and}\rforal t\in T(qA_xq).
\eneq
Note
\beq\label{MtechL-13-}
\tau\circ \psi_1(g)=\tau(q\otimes g)\rforal \tau\in T(A)\,\,\,{\rm and} \rforal g\in C(\Om)_{s.a.}.
\eneq
It follows that
\beq\label{MtechL-13}
L(g)(\tau)={\tau\circ \psi_1(g)\over{\tau(q)}}\rforal g\in C(\Om)_{s.a.}
\eneq
for all $\tau\in T(A).$

Note also that $\psi_1(v_i)\psi_2(v_i^*)\in M_K(C(X\times \Om))$ for
some integer $m\ge 1.$ By the virtue of Theorem 10 of Chapter VI of \cite{Exel1},
\beq\label{MtechLn8}
\Delta(\psi_1(v_i)\psi_2(v_i^*))&=&\Delta({\rm det}(\psi_1(v_i)\psi_2(v_i^*)))\\
&=& \Delta({\rm det}(\psi_1(v_i)){\rm det}(\psi_2(v_i^*)))
=\Delta(\psi_1({\rm det}(v_i))\psi_2({\rm det}(v_i^*)))\\
&=&\Delta(1_{qA_xq})\in \rho_{A_x}(K_0(qA_xq)).
\eneq
It follows that
\beq\label{MtechLn9}
{\rm dist}(\overline{\psi_1(v_i)\psi_2(v_i^*)}, \overline{1_{qA_xq}})=0.
\eneq
It follows from \ref{TUni} that there is a unitary $w_1\in qA_xq$ such that
\beq\label{MtechLn-10}
\|w_1\psi_2(g)w_1^*-\psi_1(g)\|<\ep/4\tforal g\in {\cal F}_1.
\eneq
There is a unitary normalizer $w_2\in A_x\cap C^*(X, \af)$ of $C(X)$ such that
$w_2pw_2^*=q.$ Note that $w_2$ has the form
$$
w_2 =\sum_{m\in \Z}u^m\chi_{\Gamma^{-1}(m)},
$$
where $\Gamma: X\to \Z$ is a continuous map.
Define $\psi_3, \psi_4, \psi_5: C(\Om)\to pA_xp$ by $\psi_3(g)=gp$ and
$\psi_4(g)=w_2^*w_1 u^{M} gpu^{*M}w_1^*w_2$ and
$\psi_5(g)=w_2^*w_1(g\circ \Psi_x)pw_1^*w_2$ for all $g\in C(\Om).$
As above, we compute that
\beq\label{MtechL-14}
[\psi_5]=[\psi_3]\,\,\, {\rm in}\,\,\, KL(C(\Om), pA_xp).
\eneq
By (\ref{MtechL-8}), the choice of ${\cal G}_1$ and $\dt_1,$
\beq\label{MtechL-15}
[\psi_5]|_{\cal P}=[\psi_4]|_{\cal P}.
\eneq
It follows from \ref{TAx} that
\beq\label{MtechL-16}
\tau\circ \psi_3(g)=\tau\circ \psi_4(g)\tforal g\in C(\Om)\andeqn\tforal \tau\in T(pA_xp).
\eneq
It is clear that
\beq\label{MtechL-17-}
\psi_3(v_i)\psi_4(v_i^*)\in CU(pAp).
\eneq
It follows that
\beq\label{MtechL-17}
{\tilde \Delta}_\tau (\overline{\psi_3(v_i)\psi_4(v_i^*)})\in \overline{\rho_A(K_0(A))}.
\eneq
Therefore, by (\ref{MtechLn-3}),
\beq\label{MtechL-18}
{\rm dist}(\overline{\psi_1(v_i)\psi_2(v^*)}, \overline{1_{qA_xq}})=0.
\eneq
By applying \ref{TUni} again, we obtain a unitary $w_3\in pA_xp$ such that
\beq\label{MtechL-19}
\|w_3gpw_3^*-\psi_4(g)\|<\ep/4\tforal g\in {\cal F}_1.
\eneq
Put $w=w_2w_3.$ Then $w\in A_x$ and
\beq\label{MtechL-20}
w^*w=pw_3^*w_2^*w_2w_3p=p=\chi_O\andeqn ww^*=w_2w_3w_3^*w_2^*=w_2pw_2^*=q=\chi_{\af^M(O)}.
\eneq
So (2) holds.
Moreover (see also (\ref{MtechL-8})),
\beq\label{MtechL-21}
\|wgpw^*-gq\| &\le & \|w_2(w_3gpw_3^*)w_2^*-w_2\psi_4(g)w_2^*\|
+\|w_2\psi_4(g)w_2^*-gq\|\\
&<&\ep/4+\|w_1u^Mgpu^{*M}w_1^*-gq\|\\\label{MtechL-20+}
&<& \ep/4+\ep/8+\|w_1\Psi_x(g)w_1^*-gq\|\\
&<& \ep/4+\ep/8+\ep/4=5\ep/8\rforal g\in {\cal F}_1.
\eneq
It follows that
\beq\label{MtchL-22}
\|wg-gw\|=\|wgp-gqw\|=\|(wgp-gqw)w^*\|=\|wgpw^*-gq\|<5\ep/8
\eneq
for all $g\in {\cal F}_1.$
Since $O\cup \af^M(O)\subset B_x,$ by (\ref{MtechL-4}), for all $f\in {\cal F}_0,$
\beq\label{MtechL-23}
\|wf-fw\|&\le &\|wpf-wpf(x)\|+\|wpf(x)-f(x)qw\|+\|f(x)qw-fw\|\\
&<&\ep/8+\|wf(x)-f(x)w\|+\ep/8=\ep/4.
%\\&<& \ep/4+5\ep/8=7\ep/8.
\eneq
Thus (4) holds.
To see (3), we note that
\beq\label{MtechL-24}
pu^i=pu\chi_{\af^{-1}(O)}u\chi_{\af^{-2}(O)}\cdots u\chi_{\af^{-i}(O)}\andeqn\\
(u^{*i}q)^*=qu\chi_{\af^{M-1}(O)}u\chi_{\af^{M-2}(O)}\cdots u\chi_{\af^{M-i}(O)}
\eneq
for $i=1,2,...,N-1.$ Since $x\in O,$ (1) implies that $pu^i$ and $u^{*i}q$ are in $A_x.$
From this one concludes that
\beq\label{MtechL-25}
u^{*i}wu^i\in A_x.
\eneq
This proves the lemma.
\end{proof}

%\begin{rem}\label{RR1}
{\rm
%In  the proof of \ref{MtechL},
%we note that
%if we assume that
%\beq\label{RR1-1}
%\overline{\rho_{A_x}(K_0(A_x))}=\overline{\rho_A(K_0(A))},
%\eneq
%then
In \ref{MtechL}, if $k=1,$
the assumption (\ref{MtechL-1}) means that
\beq\label{RR-1}
[{\tilde \phi_x}]=[{\rm id}_{C(\Om)}]\,\,\, {\rm in}\,\,\, KL(C(\Om),C(\Om)).
\eneq

We would like to state the following:

\begin{cor}\label{CMtech}
In \ref{MtechL}, in the case  $k=1,$  the lemma holds if
the condition that $U(C(\Om))=U_0(C(\Om))$ is replaced by the following:
for each $z\in U(C(\Om))/U_0(C(\Om)),$ there exists $v\in U(C(\Om))$ with
$[v]=z$
and $h\in C(X)_{s.a.}$ such that
\beq\label{CMtech-1}
{\tilde \phi_y}(v)=v\exp(ih(y))\tforal y\in X.
\eneq
\end{cor}

\begin{proof}
For each $z\in U(C(\Om))/U_0(C(\Om)),$ choose $v\in U(C(\Om))$ such that
$[v]=z$ and (\ref{CMtech-1}) holds.
Therefore the rotation map
\beq\label{CMtech-2}
\Delta_t(v(\af\times \phi)(v^*))=-t(h)+\rho_{C(X\times \Om)}(K_0(C(X\times C(\Om))))
\eneq
for all $t\in T(C(X\times \Om)).$
Since $\af$ is minimal, $C(X)\rtimes_\af\Z$ is a unital A$\T$-algebra of real rank zero. In particular,
$\rho_{B}(K_0(B))$ is dense in $\Aff(T(B)),$ where $B=C(X)\rtimes_\af\Z.$
Note that
\beq\label{CMtech-3}
\rho_B(K_0(B))=\rho_A(K^0(X,\af, \Z))\andeqn h(y)\in C(X)_{s.a.}.
\eneq
It follows that, for each $\tau\in T(A),$
\beq\label{CMtech-4}
\Delta_\tau(v(\af\times \phi)(v^*))\in \overline{\rho_A(K^0(X,\af,\Z))}.
\eneq
It follows from the last part of \ref{AKT} and (\ref{CMtech-4}) above that
\beq\label{CMtech-5}
\overline{\rho_{A_x}(K_0(A_x))}=\overline{\rho_A(K_0(A))}.
\eneq
In the proof of \ref{MtechL}, if $v_j\in U(C(\Om)),$ then
\beq\label{CMtech-6}
\psi_1(v_j)\psi_2(v_j)^*=v_jq\Psi_x(v_j^*)q=v_j{\tilde \phi_x}(v_j^*)=\exp(-ih_j)
\eneq
for some $h_j\in C(X)_{s.a.}.$
It follows that
\beq\label{CMtech-7}
D_\tau(\psi_1(v_j)\psi_2(v_j))=-\tau(h_jq)\in \overline{\rho_{A_x}(K_0(qA_xq))},
\eneq
since $h_jq\in C(X).$
If $v_j\not=U(C(\Om))$ and $v_j\in M_K(C(\Om))$ with $K>1,$
let $w_j={\rm det}v_j\in U(C(\Om))$ and
\beq\nonumber
v_j'=\begin{pmatrix} {\rm det}v_j^* &0&\cdots & 0\\
                                       0 & 1 &\cdots &0\\
                                       && \ddots  & &\\
                                         &&&1\\
                                       \end{pmatrix}v_j.
                     \eneq

   Note that ${\rm det}(v_j')=1.$
   As in the proof of \ref{MtechL},
   one has
\beq\label{CMtech-8}
\Delta(\psi_1(v_j')\psi_2((v_j')^*))\in \overline{\rho_{A_x}(K_0(qA_xq))}.
\eneq
Put $w_j'=(v_j'(v_j^*))^*.$
Then
\beq\label{CMtech-9}
\Delta(\psi_1(v_j)\psi_2(v_j^*)) &=&
\Delta(\psi_1(w_j'v_j')\psi_2((w_j'v_j')^*))\\
&=&\Delta(\psi_1(v_j')\psi_2((v_j')^*)\psi_2((w_j')^*)\psi_1(w_j'))\\
&=&\Delta(\psi_1(v_j')\psi_2((v_j')^*))+\Delta(\psi_2((w_j')^*)\psi_1(w_j'))\\
&=& \Delta(\psi_1(v_j')\psi_2((v_j')^*))-\Delta(\psi_1(w_j)\psi_2(w_j^*))
\in \overline{\rho_{A_x}(K_0(qA_xq))}.
\eneq
Thus (\ref{MtechLn9}) also holds in this. The rest of the proof is exactly the same as that of \ref{MtechL}.
\end{proof}

The following lemma was taken from the proof of Theorem 5.6 of \cite{LM2}.

\begin{lem}\label{Ltr1}
Let $(X, \af)$ be a Cantor minimal system, $\Om$ be a compact connected
finite dimensional metric space  and let $\phi: X\to {\rm Homeo}(\Om)$ be a continuous map.
Suppose that there is $x\in X$ such that,  for any $N\in \N,$ $\dt>0,$ and any finite subset ${\cal F}\subset C(X\times \Om),$ there is an integer $M>N,$  a clopen neighborhood $O$ of $x$ and partial isometry $w\in A_x$ which satisfy the following:

{\rm (1)} $\af^{-N}(O), \af^{-N+1}(O),..., O, \af(O),...,\af^M(O)$ are mutually disjoint and $\mu(O)<\dt/M$ for every $\af$-invariant probability measure $\mu;$

{\rm (2)} $w^*w=\chi_O$ and $ww^*=\chi_{\af^M(O)};$

{\rm (3)} $u^{*i}wu^i\in A_x$ for $i=0,1,...,N-1,$

{\rm (4)} $\|wf-fw\|<\ep$ for all $f\in {\cal F}.$

Then, for any $\ep>0,$ any finite subset ${\cal F}\subset A,$ there exists a projection
$e\in A_x$ satisfies the following:

{\rm (a)} $\|ea-ae\|<\ep\tforal a\in {\cal F},$

{\rm (b)} ${\rm dist}(pap, eA_xe)<\ep$ for all $a\in {\cal F}$ and

{\rm (c)} $\tau(1-e)<\ep$  for all $\tau\in T(A).$

\end{lem}

\begin{proof}
It suffices to show the  following: for any $\ep>0,$ any finite
 subset  ${\cal F}\subset C(X\times \Om)$ and any nonzero
 element $a\in A_+\setminus \{0\},$ there exists a projection
 $e\in A_x\subset A$ such that the following hold:

 (1') $\|ef-fe\|<\ep$ for all $a\in {\cal F}\cup\{u\},$

 (2') ${\rm dist}(efe, eA_xe)<\ep$ for all $f\in {\cal F}$ and

 (3') $\tau(1-e)<\ep$ for all $\tau\in T(A).$

  Without loss of generality, we may assume that ${\cal F}^*={\cal F}.$
  Choose $N\subset \N$ so that $2\pi/N<\ep.$  Put
  $$
  {\cal G}=\bigcup_{i=0}^{N-1} u^i{\cal F}u^{*i}.
  $$
    We obtain an integer $M>N,$ a clopen neighborhood $O$ of $x$ and a partial isometry
  $w\in A_x$ satisfying  (1), (2), (3) and (4).

  Put $p=\chi_O$ and $q=\chi_{\af^M(O)}.$ Define
  \beq\label{Ltr=1+n}
  P(t)=p\cos t +w\sin t\cos t+w^*\sin t \cos t +q\sin^2 t\,\,\,t\in [0, \pi/2].
  \eneq
  Then $P(0)=p$  and $P(\pi/2)=q.$ Moreover, one checks that $P(t)$ is a continuous path of projections. By (2), (3) and by the choice of ${\cal G},$ one has
  \beq\label{Ltr1-1}
  \|u^{i*} P(t)u^if-fu^{i*}P(t)u^i\|<\ep
  \eneq
  for all $t\in [0, \pi/2],$ $i=0,1,...,N-1$ and $f\in {\cal F}.$ Define
  \beq\label{Ltr1-2}
  e=1-\left( \sum_{i=0}^{M-N}u^ipu^{i*}+\sum_{i=1}^{N-1} u^{i*}P(i\pi/2N)u^i\right).
  \eneq
  Using (1) and (2),
one verifies that $e$ is a projection. By the assumption that $u^{i*}wu^i\in A_x,$ $e\in A_x.$
By (2) and the  fact that
$$
\{p, upu^*, u^2pu^{2*},...,u^{M-N}p(u^{M-N})^*, u^*P(\pi/2N)u,u^{2*}P(2\pi/2N)u^2,...,
(u^*)^{N-1}P((N-1)\pi/2N)u^{N-1}\}
$$
is a set of orthogonal projections, it
  is ready to verify that
  \beq\label{Ltr1-3}
  \|fe-ef\|<\ep\rforal f\in {\cal F}.
  \eneq
Since
\beq\label{Ltr1-4}
\|P(i\pi/2N)-P((i-1)\pi/2N)\|<\pi/N<\ep,\,\,\, i=1,2,...,N,
\eneq
one can further verify  that
\beq\label{Ltr1-5}
\|ue-eu\|<\ep.
\eneq
It is clear that $efe\in A_x$ for all $f\in C(X\times\Om).$
Note that
\beq\label{Ltr1-6}
eue=eu(1-p)e.
\eneq
Therefore $eue\in A_x.$ One also has
\beq\label{Ltr1-7}
\tau(1-e)<M\tau(p)<\ep
\eneq
for all  $\tau\in T(A).$
\end{proof}

\begin{lem}\label{ttr1=tr1}
Let $A$ be a unital simple \CA\, and $B\subset A,$ where $1_B=1_A$ and
$B$ is a unital simple AH-algebra with no dimension growth such that
$T(B)=T(A).$ Suppose $A$ has the following property:
For any $\ep>0$ and any subset ${\cal F}\subset A,$ there is a projection
$e\in B$ such that

{\rm (1)} $\|ea-ae\|<\ep\tforal a\in {\cal F},$

{\rm (2)} ${\rm dist}(eae, eBe)<\ep\tforal a\in {\cal F}$ and

{\rm (3)} $\tau(1-e)<\ep$ for all $\tau\in T(A).$

Then $A$ has tracial rank at most one.

\end{lem}

\begin{proof}
We first show  that, with the assumption, for any given $d\in A_+\setminus \{0\},$
one can require that

{\rm (3)} $1-e\lesssim d.$
%This follows an argument in the proof of Theorem 4.5 of \cite{LP}.
%It follows from the assumption (using only (1) and (2)) that $A$ has the local approximation %property of Popa (see Lemma 4.3 of \cite{LP}).
%It follows from 2.11 of \cite{LnTAF} that $A$ has (SP) property.
We may assume that $0\le d\le 1.$
Put
\beq\label{ttr1=1-1}
\sigma=\inf\{\tau(d): \tau\in T(A)\}.
\eneq
Since $A$ is simple and $T(A)$ is compact, $\sigma>0.$  Choose $\ep_0=\min\{\sigma/2, \ep/2\}.$
By the assumption, there is a projection $e_1\in B$ such that
\beq\label{ttr=1-2}
\|e_1de_1-d_1\|<\ep_0/32\andeqn \tau(1-e_1)<\ep_0/2\tforal \tau\in T(A)
\eneq
for some $d_1\in e_1Be_1.$
Since
\beq\label{ttr=1-3}
\tau(d)=\tau(e_1de_1)+\tau((1-e_1)d(1-e_1))\rforal \tau\in T(A),
\eneq
\beq\label{ttr=1-4}
\|e_1de_1\|\ge \sigma-\ep_0/2\ge \sigma/2.
\eneq
It follows that
\beq\label{ttr=1-5}
\|d_1\|\ge \sigma_2-\ep_0/32\ge 15\sigma/32.
\eneq
Put $\dt=\ep_0/32.$
Then, by Proposition 2.2 of \cite{Ro},
\beq\label{ttr=1-6}
0\not=f_\dt(d_1)\lesssim e_1de_1\sim d^{1/2}e_1d^{1/2}\lesssim d.
\eneq
Since $B$ has tracial rank at most one, $e_1Be_1$ has property (SP).
In particular, there is a non-zero projection $e_2\in \overline{ f_{\dt}(d_1)Bf_\dt(d_1)}.$
Put
\beq\label{ttr=1-7}
\ep_2=\min\{\ep/2, \inf\{\tau(e_2): \tau\in T(A)\}\}.
\eneq
Then, by the assumption,
there is a projection $e\in B$ such that

{\rm (1)} $\|ea-ae\|<\ep_2\le \ep/2$ for all $a\in {\cal F},$

{\rm (2)} ${\rm dist}(eae, eBe)<\ep_2\le \ep/2$ for all $a\in {\cal F}$ and

{\rm (3')} $\tau(1-e)<\ep_2\le \tau(e_2)$ for all $\tau\in T(A).$

Since $T(B)=T(A)$ and $B$ has tracial rank at most one,
\beq\label{ttr=1-8}
1-e\lesssim e_2\lesssim f_\dt(d_1)\lesssim d.
\eneq

Note that $B$ has tracial rank at most one. Exactly the same argument used in the proof of  Lemma 4.4 of \cite{LP} shows
that $A$ has tracial rank at most one.  Another way to reach the conclusion is to apply Lemma
4.3 of \cite{HLX}.
\end{proof}

\begin{thm}\label{MT1}
Let $(X, \af)$ be a Cantor minimal system, $\Om$ be a compact connected
finite dimensional metric space with $U(C(\Om))=U_0(C(\Om))$ and let $\phi: X\to {\rm Homeo}(\Om)$ be a continuous map.
Suppose that there exists $x\in X$ and an integer $k\ge 1$ such that
\beq\label{MT1-1}
[\Phi_y]=[{\rm id}_{C(\Om)}]\,\,\, {\rm in}\,\,\, KL(C(\Om), C(\Om)),
\eneq
where
\beq\label{MT1-2}
\Phi_y(f)=f\circ \phi_{\af^{-k}(y)}^{-1}\circ \phi_{\af^{1-k}(x)}^{-1}\circ\cdots \circ\phi_{\af^{-1}(x)}^{-1}\circ \phi_x^{-1}\tforal f\in C(\Om)
\eneq
for all $y\in \{\af^{j-1}(x): j\in \N\}$ and suppose that $\af^k$ is minimal.

If $\af\times \phi$ is minimal, then
$A=C(X\times \Om)\rtimes_{\af\times \phi} \Z$ has tracial rank at most one.
Consequently $A$ is isomorphic to a unital simple AH-algebra with no dimension growth.
\end{thm}

\begin{proof}
This follows from \ref{MtechL}, \ref{Ltr1} and \ref{ttr1=tr1}.
\end{proof}

\begin{thm}\label{MTNC}
Let $(X, \af)$ be a Cantor minimal system, $\Om$ be a compact connected
finite dimensional metric space and let $\phi: X\to {\rm Homeo}(\Om)$ be a continuous map. Suppose that
\beq\label{MTNC-1}
[\phi_y]=[{\rm id}_{C(\Om)}]\,\,\,{\rm in}\,\,\, KL(C(\Om),C(\Om))
\eneq
for all $y\in X$ and, for each $z\in U(C(\Om))/U_0(C(\Om)),$ there exists $v\in U(C(\Om))$ and $h\in C(X)_{s.a.}$ such that
\beq\label{MTNC-2}
{\tilde \phi_y}(v)=v\exp(ih(y))\tforal y\in X.
\eneq
If $\af\times \phi$ is minimal on $X\times \Om,$ then
$A=C(X\times \Om)\rtimes_{\af\times\phi}\Z$ is isomorphic to a unital simple AH-algebra with no dimension growth.
\end{thm}

\begin{proof}
This follows from \ref{CMtech}, \ref{Ltr1} and \ref{ttr1=tr1}.
\end{proof}

\section{Examples}

\begin{NN}
{\rm
Let $(X, \af)$ be a Cantor minimal system and $\phi: X\to \T^n$ be a continuous map.  
One may write $\T^n=\R/\Z\times \R/\Z\times \cdots \times \R/\Z.$ 
Given $\eta\in \T^n,$ one write $\eta=(t_1, t_2,...,t_n),$ where $t_j\in \R/\Z.$
Define $\eta(\xi)=(s_1+t_1,s_2+t_2,...,s_n+t_n),$ where $\xi=(s_1,s_2,...,s_n)\in \T^n.$
Define $(\af\times \phi)(x,\xi)=(\af(x), \phi_x(\xi))$ for all $(x, \xi)\in X\times \T^n.$
Since $X$ is totally disconnected, we may also write $\phi_x=(\exp(i\theta_1(x)), \exp(i\theta_2(x)),...,\exp(i\theta_n(x))),$
where $\theta_j\in C(X)_{s.a.}.$ 
Note for each $x\in X,$ 
$[\phi_x]=[{\rm id}_{C(\Om)}].$ Let $z$ be the standard unitary generator of $C(\T).$ Denote by  
$z_j\in C(\T^n)$ the function which maps $(s_1,s_2,...,s_n)$ to $s_j.$
Then ${\tilde \phi_y}(z_j)=z_j\exp(i\theta_j(y))$ for all $y\in X,$ $j=1,2,...,n.$
Therefore if $\af\times \phi$ is minimal, then \ref{MTNC} applies. In particular, 
when $\af\times \phi$ is minimal, $C(X\times \T^n)\rtimes_{\af\times \phi}\Z$ is a unital simple 
\CA\, with tracial rank at most one.  In the case that $n=1,$ Lemma 4.2 in \cite{LM2} 
provides a necessary and sufficient condition for $\af\times \phi$ being minimal (see also 
\cite{Su}). 
}
\end{NN}

\begin{NN}

{\rm
Let $\{m_n\}$ be a sequence of integers with $m_n\ge 2$ and $m_n|m_{n+1}.$
Let  $\lambda_n: \Z/m_{n+1}\to \Z/m_n$ be the quotient map. The the
inverse limit $\varprojlim\Z/m_n$ is the Cantor set. The so-called odometer
action $\af$ is defined by $\af(x)=x+1$ for $x\in \varprojlim\Z/m_n.$
Such an action is always minimal. Moreover the family $\{\af^k: k\in \N\}$ is equicontinuous
on the Cantor set (II.9.6.7 of \cite{dVj}).
}
\end{NN}

\begin{lem}\label{odok}
For each integer $k\ge 2,$ there exists an odometer action $\af$ on the Cantor set
such that $\af^k$ is minimal.
\end{lem}

\begin{proof}
Fix $k\ge 2.$ Choose a sequence of integers $\{m_n\}$ such that $(k,m_n)=1,$ i.e.,
$k$ and $m_n$ are relatively prime and $m_n|m_{n+1},$ $n=1,2,....$
Fix $x\in \varprojlim\Z/m_n\Z.$ We will show that $\{\af^{mk}(x): m\in \N\}$ is dense.
Let $y\in \varprojlim \Z/m_n\Z.$ Fix $\ep>0.$
Since $\{\af^{m}: m\in \N\}$ is equicontinuous,
there is $\dt>0$ such that, for any pair of $z_1, z_2\in \varprojlim \Z/m_n\Z,$
\beq\label{dodk-1}
{\rm dist}(\af^m(z_1), \af^m(z_2))<\ep/2\tforal m\in \N,
\eneq
provided ${\rm dist}(z_1, z_2)<\dt.$

There is an integer $j\ge 1$ and  $x',y'\in \Z/m_j\Z$
such that $x_0=\{x_n'\}, y_0=\{y_n'\}\in \varprojlim \Z/m_n\Z$
and $x_n'=\gamma_{j,n}(x')$ and $y_n'=\gamma_{j,n}(y')$ for all $n< j,$
where $\gamma_{j,n}=\gamma_n\circ \gamma_{n+1}\cdots \circ \gamma_j$ and
such that
\beq\label{odok-2}
{\rm dist}(x_0, x)<\dt\andeqn{\rm dist}(y_0, y)<\dt.
\eneq
We may assume that $\dt<\ep/2.$
Since $(k, m_j)=1,$ there is $m\in \N$ such that $mk \equiv 1(m_j)$ or
$mk\equiv -1(m_j).$  Since $-(m_j-1)\equiv 1(m_j),$ in fact, in both case,
there is an integer $l_1\ge 1$ such that $l_1k\equiv 1(m_j).$
 We may assume that $y'=x'+m$ in $\Z/m_j\Z.$  Then $y'=x'+ml_1k$ in $\Z/m_j\Z.$
 Then one computes that
 \beq\label{odok-3}
 \af^{ml_1k}(x_0)=x_0+ml_1k=y_0.
 \eneq
It follows
that
\beq\label{odok-4}
{\rm dist}(\af^{ml_1k}(x), y) &\le & {\rm dist}(\af^{ml_1k}(x), \af^{ml_1k}(x_0))+
{\rm dist}(\af^{ml_1k}(x_0), y)\\
&<& \ep/2+{\rm dist}(y',y)<\ep.
\eneq
\end{proof}

The following is a result of Karen Strung (Proposition 2.1 and Section 5 of \cite{KS})  which we quote here for convenience. Note, if $\Om$ is connected, $\bt^m$ is minimal for any
non-zero integer $m.$

\begin{prop}\label{KS}
Let $\af$ be an odometer action on the Cantor set and $\Om$ is a  compact metric space. Suppose that $\bt: \Om\to \Om$ is a minimal homeomorphism
such that $\bt^m$ is minimal for all $m\in \N.$ Then
$\af\times \bt$ is a minimal homeomorphism on $X\times \Om.$
\end{prop}

\begin{exm}\label{ExOddsph}
{\rm
Let $\bt: S^{2n+1}\to S^{2n+1}$ ($n\ge 1$) be a minimal
homeomorphism. It is known such $\bt$ exists. Fathi and Herman (\cite{FH})
showed that there exists a unique ergodic and minimal diffeomorphism
on $S^{2n+1}.$ The group $\R/\Z$ can act on $S^{2n+1}$ freely by rotations. By a result of A. Windsor, there are minimal homeomorphisms $\bt$
on $S^{2n+1}$ such that $\bt$ can have any number of ergodic measures (\cite{Wa}). It follows from \ref{KS} that $\af\times \bt$ are minimal
homeomorphisms on $X\times S^{2n+1},$ where $\af$ is an odometer which has many invariant probability measures.
}
\end{exm}

\begin{cor}\label{CS3}
Let $\af$ be an odometer on the Cantor set $X$ and
$(S^{2n+1}, \bt)$ be a minimal dynamical system with $n\ge 1.$
Then $\af\times \bt$ is minimal and $A=C(X\times S^{2n+1})\rtimes_{\af\times \bt}\Z$ has tracial rank at most one.
\end{cor}

\begin{proof}
It follows from \ref{ExOddsph} that $\af\times \bt$ is minimal.
Since $\bt$ is minimal, it does not have fixed point. Therefore
$\bt$ has zero degree. If follows that $[\bt]=[{\rm id}]$ in $KK(C(S^{2n+1}), C(S^{2n+1})).$ Moreover, $U(C(S^{2n+1}))=U_0(C(S^{2n+1})).$ Thus Theorem \ref{MT1} applies (with $k=1$).
\end{proof}

\begin{exm}\label{ExPR3}
{\rm
Consider an $\R/\Z$ action on $RP^{2n+1}.$ We identify
$RP^{2n+1}$ as $ SO(2n).$
Define $\gamma: \R/\Z\to SO(2n)$ by
\beq\label{ExPR3-1}
\gamma(t)=\begin{pmatrix} \cos(\pi t/2) & \sin(\pi t/2) & 0&&\\
                            -\sin(\pi t/2) & \cos(\pi t/2) & 0&&\\
                            0& 0&1 &&\\
                            &&& \ddots\\
                             &&&& 1
                            \end{pmatrix}.
\eneq
Define an action $\R/\Z\times SO(2n)\to SO(2n)$ by
\beq\label{Exrp3-2}
\Gamma(t)(x)=\gamma(t)x\rforal t\in \R/\Z\andeqn  x\in SO(2n).
\eneq
It is clear $\Gamma$ is free and a $C^{\infty}$ -diffeomorphism.
For each $1/n^2>\dt>0$ and $r\in \Q/\Z,$  by \cite{Wa}, there is
a minimal diffeomorphism $\bt_r: SO(2n)\to SO(2)$ such
that
\beq\label{Exrp3-3}
{\rm dist}(\bt_r(x),\gamma(r)x)<\dt\rforal x\in SO(2n).
\eneq
}
\end{exm}

\begin{cor}\label{RP3T1}
Let $n\ge 1$ be an integer.
There are odometer actions  $\af$  on the Cantor
set  $X$ such that for any
minimal homeomorphism $\bt$ on $RP^{2n+1},$
$A=C(X\times RP^{2n+1})\rtimes_{\af\times \bt}\Z$ is
a unital simple \CA\, with tracial rank at most one.
\end{cor}

\begin{proof}
First it is well known that  $H^1(C(RP^{2n+1}),\Z)=\{0\}.$ In other words,
$U(C(RP^{2n+1}))=U_0(C(RP^{2n+1})).$

 Note that
\beq\label{RP3T-2}
K_0(C(RP^{2n+1}))=\Z\oplus G_0\andeqn K_1(C(RP^{2n+1}))=\Z,
\eneq
where $G_0$ is a finite group such that $2g=0$ for all $g\in G_0.$
Any automorphism on $K_0(C(RP^{2n+1}))$ induced by an automorphism
on $C(RP^{2n+1})$ has the form
\beq\label{RP3T-3}
\begin{pmatrix} {\rm id}_{\Z} &0\\
                         \phi_{2,1} &\phi_{2,2}\end{pmatrix},
\eneq
where $\phi_{2,1}:\Z\to G_0$ and $\phi_{2,2}: G_0\to G_0$ are \hm s,
 since it sends identity of $C(RP^{2n+1})$ to itself and $G_0$ is finite.
 Automorphisms of the form in (\ref{RP3T-3}) is a finite subgroup. 
 Suppose that the order of the group is $k_1.$ Then, for any automorphism
 $\phi: C(RP^{2n+1})\to C(RP^{2n+1}),$
 \beq\label{RP3T-4}
 \phi_{*0}^{k_1}={\rm id}_{K_0(C(RP^{2n+1}))}.
 \eneq

We note that $H_0(RP^{2n+1}, \Q)=\Q,$  $H_{2n+1}(RP^{2n+1}, \Q)=\Q$
and $H_i(RP^{2n+1})=\{0\}$ for all other $i.$ Also
$H_0(RP^{2n+1}, \Z)=\Z$ and $H_{2n+1}(RP^{2n+1}, \Z)=\Z.$
Let ${\tilde \bt}: C(RP^{2n+1})\to C(RP^{2n+1})$ be the isomorphism induced by $\bt.$
Note that $\bt_*={\rm id}$ on $H_0(RP^{2n+1}, \Z)=\Z$
and ${\tilde \bt}_*(1)=\pm 1$ on $H_{2n+1}(RP^{2n+1}, \Z)=\Z.$
Let
\beq\label{RP3T-5}
L_\bt&=&\sum_{k\ge 0}(-1)^k Tr(\bt_*|(H_k(RP^{2n+1}, \Q)))\\
&=&Tr({\rm id}|(H_0(RP^{2n+1},\Q))+(-1)^{2n+1}Tr(\bt_*|(H_{2n+1}(RP^{2n+1}, \Q))))
\eneq
be the Lefschetz number.
If
$\bt$ is minimal, it does not have fixed point. So
$L_\bt=0.$ It follows that $\bt_{*}(1)=1$ on $H_{2n+1}(RP^{2n+1},\Z).$
It follows that, for any minimal homeomorphism $\bt$ on $RP^{2n+1},$
\beq\label{RP3T-6}
(\phi_{\bt})_{*1}={\rm id}_{K_1(C(RP^{2n+1}))}.
\eneq
We compute that
\beq\label{RP3T1-3}
K_0(C(RP^{2n+1}, \Z/2\Z))=\Z/2\Z\oplus G_0\andeqn\\
0\to \Z/2\Z\to K_1(C(RP^{2n+1}, \Z/2\Z))\to G_0\to 0.
\eneq
Let  $k_2$ be the order of ${\rm Aut}(\Z/2\Z\oplus G_0)$ (which is finite).
One also checks, from the  above,  $K_1(C(RP^{2n+1}),\Z/2\Z)$ is a finite
abelian group such that $4x=0$ for all $x\in K_1(C(RP^{2n+1}),\Z/2\Z).$
Let $k_3$ be the order of ${\rm Aut}(K_1(C(RP^{2n+1}))).$

Define $k=k_1\cdot k_2\cdot k_3$ which depends on $n$ only.
Choose a sequence of integers $\{m_j\}$ such that $m_j|m_{j+1}$ for all $j$ and
each $m_j$ is prime relative to $k.$ Then, by \ref{odok}, there are
odometer actions $\af$ on the Cantor set such that $\af^k$ is also minimal.

Now let $\bt$ be a minimal homeomorphism on $RP^{2n+1}.$
Then, by above,
\beq\label{RP3T-8}
[({\tilde \bt})^k]|_{K_i(C(RP^{2n+1}))}={\rm id},\,\,\,i=0,1\andeqn
[{\tilde \bt})^k]|_{K_i(C(RP^{2n+1}),\Z/2\Z)}={\rm id},\,\,\, i=0,1.
\eneq
Note that
\beq\label{RP3T1-4}
KL(C(RP^{2n+1}), C(RP^{2n+1}))&=& Hom_{\Lambda} \underline{K}(C(RP^{2n+1}), C(RP^{2n+1}))
\eneq
It follows from  2.11 of \cite{DL} that, to check $[{\tilde \bt}^k]=[{\rm id}],$ it suffices to
show that (\ref{RP3T-8}) holds since $2g=0$ for all $g\in G_0.$
Therefore
\beq\label{RP3T1-5}
[{\tilde \bt}^k]=[{\rm id}]\,\,\,{\rm in}\,\,\, KL(C(RP^{2n+1}), C(RP^{2n+1})).
\eneq
By the assumption we also have
that $\af^k$ is minimal.
Hence \ref{MT1} applies to $\af\times \bt.$

\end{proof}

\section{Applications}

In this section we consider
$A=C(\Om)\rtimes_\bt \Z,$ where $\Om$ is a connected compact metric space and $\bt$ is a minimal homeomorphism on $\Om.$
Specific examples are the cases that $\Om=S^{2n+1}$ or $\Om=RP^{2n+1},$ where $n\ge 1.$ It should be noted that
there are no minimal homeomorphisms on even spheres or even dimensional
real projective spaces. Our results can also apply to other
connected spaces.

\begin{thm}\label{MT2}
Let $\Om$ be a connected compact metric space with finite covering dimension such that
$U(C(\Om))=U_0(C(\Om))$ and let $\bt: \Om\to \Om$ be a minimal homeomorphism. Suppose that $[{\tilde \bt}^k]=[{\rm id}]$ in $KL(C(\Om), C(\Om))$ for some integer $k\ge 1,$ where
${\tilde \bt}(f)=f\circ \bt^{-1}$ for all $f\in C(\Om).$
Then $A=C(\Om)\rtimes_\bt\Z$ has rational tracial rank at most one, i.e.,
$A\otimes U$ has tracial rank at most one for any infinite dimensional UHF-algebra $U.$
In particular,  $A$ is in ${\cal A}.$
\end{thm}

\begin{proof}
First we note that since $\Om$ has finite covering dimension, it follows from \cite{TW} that
$A$ has finite nuclear dimension.
Let $\af$ be an odometer action on the Cantor set such that
$\af^k$ is also minimal.
It follows from \ref{KS}  that $\af\times \bt$ is a minimal action. Let $B=C(X\times \Om)\rtimes_{\af\times\bt}\Z.$

It follows from
\ref{MT1} that $B$ has tracial rank at most one.
Consider the embedding $\imath: A\to B$ that sends
 $C(\Om)\to C(X\times \Om)$ and sends implementing
unitary to the implementing unitary in a natural way.
Any tracial state $\tau$ of $B$ is given by an
$\af\times \bt$-invariant Borel probability measure.
Let $\tau_0$ be the unique tracial state on $C(X)\rtimes_{\af}\Z$ which is given
by the $\af$-invariant Borel probability measure. Then each $\af\times \bt$ -invariant
tracial state on $C(X\times \Om)=C(X)\otimes C(\Om)$  has the form $\tau_0\otimes \tau_1,$ where $\tau_1$ is a $\bt$-invariant tracial state on $C(\Om).$  It follows
 that  the map $\imath_{\sharp}: T(B)\to T(A)$ induced by $\imath$ is a homeomorphism.
 It follows from \ref{Ltrace}  that $\rho_B(K_0(B))=\rho_B(K^0(X,\af, K_0(C(\Om)))).$
 Note also since $\Om$ is connected, $\rho_{C(\Om)}(K_0(C(\Om)))=\Z.$
 Therefore if $\tau_0\otimes \tau_1'$ and $\tau_0\otimes \tau_2'$ are two tracial states
 then they induce the same state on $K^0(X, \af, K_0(C(\Om))).$
 It follows that they induce the same state on $K_0(B).$  It follows from
 Theorem 4.2 of \cite{W2} that $A\otimes U$ has tracial rank at most one. It follows from \cite{LS}
 that $A\otimes U$ has tracial rank at most one for all infinite dimensional UHF-algebras $U.$
 Since $A=C(\Om)\rtimes_\bt \Z,$ it satisfies the Universal Coefficient Theorem. Furthermore,
 by \cite{TW}, $A$ is ${\cal Z}$-stable. Therefore $A\in {\cal A}.$ 

\end{proof}

\begin{thm}\label{MT3}
Let $\Om$ be a connected compact metric space with finite covering dimension such
that $H^1(\Om, \Z)=\{0\}$ and  $K_i(C(\Om))=\Z\oplus G_{i},$ where $G_i$ is a finite
group. Suppose that $\bt: \Om\to \Om$ is a minimal homeomorphism.
Then $A=C(\Om)\rtimes_\bt\Z$ has rational tracial rank at most one and
is in ${\cal A}.$
\end{thm}

\begin{proof}
This is a corollary of \ref{MT2}. We note that $U(C(\Om))=U_0(C(\Om)).$
Therefore it suffices to show that
$[{\tilde \bt}^k]=[{\rm id}]$ in $KL(C(\Om)).$
Similar to  the proof of \ref{RP3T1}, it is easy to see that there exists an integer $k_i\ge 1$ such that
$(({\tilde \bt})^{k_i})_{*i}=({\tilde \bt})_{*i}^{k_i}={\rm id}_{K_i(C(\Om))},$ $i=0,1.$

Let $r_i$ be the order of $G_i.$ For each $1\le j\le (r_i)!,$ there exists a short exact sequence
\beq\label{MT3-1}
0\to \Z/j\Z\oplus G_i/jG_i\to K_i(C(\Om),\Z/j\Z)\to G_i^{(j)}\to 0,
\eneq
where $G_i^{(j)}=\{g\in K_i(C(\Om)): jg=0\},$ $i=0,1.$
Therefore $K_i(C(\Om),\Z/j\Z)$ is a finite group, $i=0,1.$
Note
\beq\label{MT3-2}
[{\tilde \bt}]|_{K_i(C(\Om), \Z/j\Z)}\in {\rm Aut}(K_i(C(\Om), \Z/j\Z)).
\eneq
However,  ${\rm Aut}(K_i(C(\Om), \Z/j\Z))$ is a finite group.  Therefore, for some
$m_{i,j}\ge 1,$
$$[{\tilde \bt}^{m_{i,j}}]|_{K_i(C(\Om), \Z/j\Z)}={\rm id}_{K_i(C(\Om), \Z/j\Z)},\,\,i=0,1.$$
Put
\beq\label{MT3-3}
k=k_1\cdot k_2\cdot \prod_{1\le j\le (r_i)!, \, i=0,1}m_{i,j}.
\eneq
One checks that
\beq\label{MT3-4}
({\tilde \bt}^k)_{*i}={\rm id}_{K_i(C(\Om))}\andeqn [{\tilde \bt}^k]|_{K_i(C(\Om), \Z/j\Z)}={\rm id}_{K_i(C(\Om), \Z/j\Z)},
\eneq
$j=1,2,...,(r_i)!,$ $i=0,1.$
Since $r_i$ is the order of $G_i,$ by 2.11 of \cite{DL},
\beq\label{MT3-5}
[{\tilde \bt}^k]=[{\rm id}_{C(\Om)}].
\eneq
\end{proof}

\begin{cor}\label{MCS3}
Let $\bt$ be a minimal homeomorphism on $S^{2n+1}$ ($n\ge 1$).  Then
$A=C(S^{2n+1})\rtimes_\bt \Z$ has rational tracial rank at most one and is in ${\cal A}.$
\end{cor}

\begin{proof}
We have noted that $U(C(S^{2n+1}))=U_0(C(S^{2n+1}))$  and any minimal homeomorphism
$\bf$ has the property $[\bt]=[{\rm id}].$ So \ref{MT2} applies.
%That $A$ has rationally tracial rank at most one follows immediately from the above theorem and the proof of
%\ref{CS3}. To see it is in the classifiable class, we note that $A$ satisfies the Universal Coefficient Theorem ($i=1,2$). It follows from \cite{TW} that $A_i$ is ${\cal Z}$-stable.
%Therefore the result in \cite{Lninv} applies.
\end{proof}

\begin{cor}\label{CClass1}
Let $\bt_1,\bt_2: S^{2n+1}\to S^{2n+1}$ ($n\ge 1$) be two minimal homeomorphisms and let
$A_i=C(S^{2n+1})\rtimes_{\bt_i} \Z,$ $i=1,2.$
Then $A_1\cong A_2$ if and only if  $T(A_1)\cong T(A_2).$
\end{cor}

\begin{proof}
One computes, by Pimsner-Voiculescu exact sequence (\cite{PV}) that
$K_i(A_j)=\Z\oplus \Z$ for  $i=0,1$ and $j=1,2.$  One also computes that
the order of $K_0(A_j)$ is determined by one copy of $\Z$ from
the rank of projections of $M_k(C(S^{2n+1}))$ for all $k$ and
$K_0(A_1)$ and $K_0(A_2)$ are unital order isomorphic.  Furthermore,
all traces agree on $K_0(A_1)=K_0(A_2).$
Therefore their Elliott invariant
is determined by $T(A_i),$ $i=1,2.$  Now, by \ref{MCS3},  the classification theorem in \cite{Lninv} applies.
\end{proof}

\begin{cor}\label{MCRP3}
Let $\bt$ be a minimal homeomorphism on $RP^{2n+1}$ (for $n\ge 1$).
Then
$A=C(RP^{2n+1})\rtimes\Z$ has rational tracial rank at most one and is in ${\cal A}.$
\end{cor}

\begin{proof}
We have noted that $U(C(RP^{2n+1}))=U_0(C(RP^{2n+1}))$ and
$K_0(C(RP^{2n+1}))=\Z\oplus G,$ where $G$ is a finite group and
$K_1(C(RP^{2n+1}))=\Z.$ Thus \ref{MT3} applies.
%This follows from \ref{MT2} and the proof of \ref{RP3T1} and the proof of \ref{MCS3}.
\end{proof}

\begin{cor}\label{CCRP3}
Let $\bt_1$ and $\bt_2$ be two minimal homeomorphisms on $RP^{2n+1}$ (for $n\ge 1$) and let
$A_i=C(RP^{2n+1})\rtimes_{\bt_i}\Z,$ $i=1,2.$
Then $A_1\cong A_2$ if and only if 
\beq\label{CCRP3-1}
K_1(A_1)\cong K_1(A_2),  \,\,\,(\bt_1)_*=(\bt_2)_*\,\,\, {\rm on}\,\,\,K_0(C(RP^{2n+1}))\andeqn
T(A_1)=T(A_2).
\eneq
\end{cor}

\begin{proof}
By \ref{MCRP3}, it suffices to show that $A_1$ and $A_2$ have the same
Elliott invariant.  The assumption shows that
\beq\label{CCRP3-2}
&&K_0(C(RP^{2n+1}))/\{z-z\circ (\bt_1)_*: z\in K_0(C(RP^{2n+1}))\}\\
&\cong &
K_0(C(RP^{2n+1}))/\{z-z\circ (\bt_2)_*: z\in K_0(C(RP^{2n+1}))\}\cong \Z\oplus G_0',
\eneq
where $G_0'$ is a quotient of ${\rm Tor}(K_0(C(RP^{2n+1})).$
Moreover they are order isomorphic. By Pimsner-Voiculescu exact sequence,
we may write
\beq\label{CCRP3-3}
K_0(A_1)=(\Z\oplus G_0')\oplus \Z\cong K_0(A_2).
\eneq
Since $H^1(RP^{2n+1},\Z)=\{0\},$ it follows that
\beq\label{CCRP3-4}
\rho_{A_1}(K_0(A_1))=\rho_{A_1}(\Z\oplus G_0')=\Z\andeqn \rho_{A_2}(K_0(A_2))=\Z.
\eneq
It follows that $K_0(A_1)$ and $K_0(A_2)$ are unital order isomorphic.
Since all traces of $A_i$ agree on $K_0(A_i),$ $i=1,2.$ It follows that
$A_1$ and $A_2$ have isomorphic Elliott invariant. Thus \cite{Lninv} applies.
\end{proof}


\begin{thebibliography}{10}

\bibitem{Bl} B. Blackadar, {\em K-theory for operator algebras},  Mathematical Sciences Research Institute Publications, 5. Springer-Verlag, New York, 1986. viii+338 pp. ISBN: 0-387-96391-X.

\bibitem{DL} M.   Dadarlat and T.  Loring, {\em A universal multicoefficient theorem for the Kasparov groups},  Duke Math. J. {\bf 84} (1996),  355--377.

\bibitem{dVj} J. de Vries,  {\em Elements of topological dynamics},  Mathematics and its Applications, 257. Kluwer Academic Publishers Group, Dordrecht, 1993. xvi+748 pp. ISBN: 0-7923-2287-8.

 \bibitem{EE} G. A.  Elliott and E. D. Evans,  {\em The structure of the irrational rotation \CA s},  Ann. of Math.  {\bf 138}  (1993),  477--501.
 
\bibitem{Exel1} R.  Exel, {\em Rotation numbers for automorphisms of \CA s, }  Pacific J. Math. {\bf 127} (1987),  31--89.



\bibitem{FH} A. Fathi and M.  Herman, {\em Existence de diffeomorphismes minimaux},  (French) Dynamical systems, Vol. I  \, Warsaw, pp. 37--59. Asterisque, No. 49, Soc. Math. France, Paris, 1977.

\bibitem{GPS} T. Giordano, I. Putnam and C. Skau, {\em Topological orbit equivalence and $C^*$-crossed products}, J. Reine Angew. Math. {\bf 469} (1995), 51--111.

\bibitem{HPS} R. Herman, I.  Putnam and C.  Skau, {\em Ordered Bratteli diagrams, dimension groups and topological dynamics}, Internat. J. Math. {\bf 3} (1992), 827--864.

  \bibitem{HLX} S.  Hu, H. Lin and Y. Xue, {\em The tracial topological rank of extensions of \CA s},  Math. Scand. {\bf 94} (2004),  125--147.

 \bibitem{Lninv} H. Lin, {\em Asymptotic unitary equivalence and classification of simple amenable \CA s},  Invent. Math. {\bf 183} (2011),  385--450.

 \bibitem{Lnappen} H. Lin, {\em    Localizing the Elliott conjecture at strongly self-absorbing \CA s, II, an appendix,} J. Reine Angew. Math., to appear
     (arXiv:0709.1654).

 \bibitem{LnHom} H. Lin, {\em Homomorphisms from AH-algebras,}  prepirnt
     (arXiv:1102.4631).

\bibitem{Lnloc} H. Lin, {\em On local AH algebras},  Memoir. Amer. Math. Soc., to appear (arXiv:1104.0445).

\bibitem{LM2} H. Lin and H. Matui, {\em  Minimal dynamical systems on the product of the Cantor set and the circle},  Comm. Math. Phys. {\bf 257} (2005),  425--471.

\bibitem{LM3} H.  Lin and H. Matui, {\em Minimal dynamical systems on the product of the Cantor set and the circle. II},  Selecta Math. (N.S.) {\bf 12} (2006), 199--239.


\bibitem{LN} H. Lin and  Z. Niu, {\em  Lifting KK-elements, asymptotic unitary equivalence and classification of simple \CA s}, Adv. Math. {\bf 219} (2008), 1729--1769.

\bibitem{LN2} H. Lin and Z. Niu, {\em The range of a class of classifiable separable simple amenable \CA s},  J. Funct. Anal.  {\bf 260} (2011),  1--29.

\bibitem{LS} H. Lin and W. Sun, {\em Tensor products of classifiable \CA s,}  preprint, (arXiv:1203.3737).

    \bibitem{LP} H. Lin and N. C.  Phillips, {\em  Crossed products by minimal homeomorphisms},  J. Reine Angew. Math. {\bf 641} (2010), 95--122.

\bibitem{M3} H. Matui, {\em Approximate conjugacy and full groups of Cantor minimal systems}, Publ. Res. Inst. Math. Sci. {\bf 41} (2005), 695--722.

\bibitem{PV} M. Pimsner and D. Voiculescu, {\em  Exact sequences for K-groups and Ext-groups of certain cross-product \CA s}, J. Operator Theory {\bf 4} (1980), 93--118.

\bibitem{Pu1} I. Putnam, {\em The \CA s associated with minimal homeomorphisms of the Cantor set,} Pacific J. Math. {\bf 136} (1989),  329--353.

\bibitem{Ro} M. R\o rdam, {\em  On the structure of simple  \CA s tensored with a UHF-algebra. II},  J. Funct. Anal. {\bf 107} (1992), 255--269.

\bibitem{KS} K. Strung, {\em \CA s of minimal dynamical systems of the product of a Cantor set and an odd dimensional sphere}, prepirnt ( arXiv:1403.3136).

\bibitem{Su} W. Sun, {\em  Crossed product \CA s of minimal dynamical systems on the product of the Cantor set and the torus}, J. Funct. Anal.,
    {\bf 265} (2013),  1105--1169.

\bibitem{KT} K. Thomsen, {\em Traces, unitary characters and crossed products by Z},  Publ. Res. Inst. Math. Sci. {\bf 31} (1995), 1011--1029.

\bibitem{TW} A.  Toms and W.  Winter, {\em Minimal dynamics and the classification of \CA s},  Proc. Natl. Acad. Sci. USA  {\bf 106} (2009), 16942--16943.

\bibitem{Wa} A. Windsor, {\em Minimal but not uniquely ergodic diffeomorphisms. Smooth ergodic theory and its applications} (Seattle, WA, 1999), 809--824, Proc. Sympos. Pure Math., 69, Amer. Math. Soc., Providence, RI, 2001.

\bibitem{W1} W. Winter, {\em Localizing the Elliott conjecture at strongly self-absorbing \CA s},  J. Reine Angew. Math., to appear ( arXiv:0708.0283)

\bibitem{W2} W. Winter, {\em Classifying crossed product $C^*$-algebras}, prepirnt
(arXiv:1308.5084).

\end{thebibliography}
\end{document}